\documentclass[reqno]{amsart}
\usepackage{amsmath, amssymb, amsthm, epsfig}
\usepackage{hyperref, latexsym}
\usepackage{url}
\usepackage[mathscr]{euscript}
\usepackage{enumerate}
\usepackage{color}
\usepackage{fullpage} 
\usepackage{setspace}

\def\today{\ifcase\month\or
	January\or February\or March\or April\or May\or June\or
	July\or August\or September\or October\or November\or December\fi
	\space\number\day, \number\year}

\DeclareMathOperator{\supp}{\mathrm{supp}}

\newtheorem{theorem}{Theorem}

\newtheorem{lemma}[theorem]{Lemma}
\newtheorem{proposition}[theorem]{Proposition}
\newtheorem{corollary}[theorem]{Corollary}
\theoremstyle{definition}

\theoremstyle{remark}
\newtheorem{remark}[theorem]{Remark}

\newcommand{\C}{\mathbb{C}}
\newcommand{\R}{\mathbb{R}}

\newcommand{\h}{\frac12}
\newcommand{\hh}{\tfrac12}
\newcommand{\ds}{\text{\rm d}s}

\newcommand{\dt}{\text{\rm d}t}
\renewcommand{\d}{\text{\rm d}}

\newcommand{\dx}{\text{\rm d}x}

\newcommand{\dy}{\text{\rm d}y}

\newcommand{\ov}{\overline}

\newcommand{\im}{{\rm Im}\,}
\newcommand{\re}{{\rm Re}\,}

\begin{document}
	\title{An extremal problem and inequalities for entire functions of exponential type}
 	\thanks{Research supported by the Brazilian Science Foundations FAPESP under Grants 2016/09906-0, 2019/12413-4 and  2021/13340-0 and CNPq under Grant 309955/2021-1, the 
 Bulgarian National Research Fund through Contract KP-06-N62/4, and  the Austrian Science Fund (FWF) via Project P-35322. MS is supported by the grant \textit{Juan de la Cierva incorporaci\'on} 2019 IJC2019-039753-I, the Basque Government through the BERC 2022-2025 program, by the Spanish State Research Agency project PID2020-113156GB-100/AEI/10,13039/501100011033 and through BCAM Severo Ochoa excellence accreditation SEV-2023-2026.}
	\subjclass[2010]{42A38, 30D15, 41A17}
	\keywords{One-delta problem, extremal problem, extremal function, entire function of exponential type}
 	\author{Andr\'es Chirre}
 	\address{Departamento de Ciencias - Secci\'on Matem\'aticas, Pontificia Universidad Cat\'olica del Per\'u, Lima, Per\'u}
 \email{cchirre@pucp.edu.pe}
 	\author[D. K. Dimitrov]{Dimitar K. Dimitrov}
 	\address{Departamento de Matem\'atica, IBILCE, Universidade Estadual Paulista UNESP, S\~{a}o Jos\'e do Rio Preto 15054, Brazil}
 	\email{d\underline{\space}k\underline{\space}dimitrov@yahoo.com}
 	\author[E. QUESADA-HERRERA]{EMILY QUESADA-HERRERA}
 	\address{Graz University of Technology, Institute of Analysis and Number Theory, Kopernikusgasse 24/II, 8010 Graz, Austria}
 	\email{quesada@math.tugraz.at}
 	\author{Mateus Sousa}
 	\address{BCAM - Basque Center for Applied Mathematics, Alameda de Mazarredo 14, 48009 Bilbao, Bizkaia, Spain}
 \email{mcosta@bcamath.org}

	\allowdisplaybreaks
	\numberwithin{equation}{section}

	\begin{abstract}
		We study two variations of the classical one-delta problem for entire functions of exponential type, known also as the Carath\'eodory--Fej\'er--Tur\'an problem. The first variation imposes the additional 
		requirement that the function is radially decreasing while the second one is a generalization which involves derivatives of the entire function. Various interesting inequalities, inspired by results due to Duffin and Schaeffer, Landau, and Hardy and Littlewood, are also established.
	\end{abstract}
	
	\maketitle

\section{Introduction}

	In the present note we study some extremal problems concerning certain quantities over specific families of entire functions of exponential type. For $\Delta>0$, we say that an entire function $G:\C \to \C$ has exponential type at most $2\pi\Delta$ if, for all $\varepsilon>0$, there exists a positive constant $C_\varepsilon$ such that $$|G(z)|\leq C_\varepsilon \,e^{(2\pi\Delta+\varepsilon)|z|}, \,\,\,\,\,\, \mbox{for all} \,\,\,\, z\in\C.$$ 
We adopt the usual convention that an entire function $f:\C\to\C$ is said to be real if its restriction to $\R$ is real-valued, as well as, that the function $g^*(z)$ is defined by $g^*(z)=\overline{g(\overline{z})}$. For $f,g \in L^1(\R)$ we denote by $f \ast g$ their convolution, which is defined by $(f \ast g)(x)=\displaystyle\int_{-\infty}^{\infty}f(y)g(x-y)\,\dy$.

\subsection{The one-delta problem} The classical \emph{one-delta problem} is to determine the infimum
$$\mathcal{A} = \inf_{\substack{g\in \mathcal{G}}}{\int_{-\infty}^{\infty} g(x)\,\dx},$$
where the class $\mathcal{G}$ consists of real entire functions $g:\C\to\C$ of exponential type at most $2\pi$ which are majorants of a \emph{one-delta function} at the origin over the real line, i.e, $g(x)\geq 0$ for all $x\in\R$ and $g(0)\geq 1$. By scaling, this is equivalent 
$$
\mathcal{A}=\inf_{\substack{f\in \mathcal{F}\\f(0)\neq  0}}\dfrac{\|f\|_1}{f(0)},
$$
where the family $\mathcal{F}$ consists of real entire functions $f:\C\to\C$ of exponential type at most $2\pi$ such that $f\in L^{1}(\R)$, and $f(x)\geq 0$ for all $x\in\R$.  This is a classical problem, and several of its variations are named after Carath\'eodory, Fej\'er and Tur\'an. We refer to \cite{Gor2001, Gor2018, Sz2015, Sz2011} for comprehensive information about its history and for some recent contributions. It is known that $\mathcal{A}=1$, and the unique extremal solution of the one-delta problem is the Fej\'er kernel, given by
	\begin{equation} \label{4_8_1:38pm}
		K(z)=\Big(\dfrac{\sin \pi z}{\pi z}\Big)^2.
	\end{equation} 
To obtain an equivalent formulation of this problem, we may consider a decomposition
result due to Krein  \cite[p. 154]{Krein}. It states that if $f:\C\to\C$ is an entire function of exponential type at most $2\pi$ such that $f\in L^1(\R)$ and $f(x)\geq 0$ for all $x\in\R$, then there exists an entire function $g:\C\to\C$ in the Paley--Wiener space $PW^2$ such that $f(z)=g(z)g^*(z)$. Here, $PW^2$ is the subspace of $L^2(\R)$ consisting of entire functions of exponential type at most $\pi$. Therefore, the one-delta problem can also be stated as finding
\begin{align}  \label{17_6pm}
\mathcal{A}=\inf_{\substack{g\in PW^{2}\\g(0)\neq 0}}\dfrac{\|g\|^2_2}{|g(0)|^2}.
\end{align} 
Other $L^p-$variations of this problem have also been studied in \cite{CMS, BCCS, Korevaar}.  Note that \eqref{17_6pm} can be stated in yet another alternative way as follows: the inequality 
\begin{align} \label{9_59pm}
1\leq \int_{-\infty}^\infty|g(x)|^2\,\dx,
\end{align} 
holds for every $g\in PW^2$ such that $g(0)=1$, and \eqref{9_59pm} reduces to an equality if and only if 
$$
g(z)=\dfrac{\sin \pi z}{\pi z}.
$$ 
Our main goal is to study some natural variations of each of the above versions of the one-delta problem.

\subsection{Monotone-delta problem}	
The \emph{monotone-delta problem} is to find 
\begin{align}\label{6_53pm}
\mathcal{A}_1=\inf_{\substack{f\in \mathcal{F}_1\\f(0)\neq  0}}\dfrac{\|f\|_1}{f(0)},
\end{align}
where the family $\mathcal{F}_1$ consists of real entire functions $f:\C\to\C$ of exponential type at most $2\pi$, such that $f\in L^{1}(\R)$, $f(x)\geq 0$ for all $x\in\R$, and $f$ is \emph{radially decreasing}, that is, $f$ is increasing on $(-\infty,0)$ and decreasing on $(0,\infty)$. To the best of our knowledge, this problem was posed explicitly by Jeffrey Vaaler. It is a variant, with additional constraints, of another problem, solved by Holt and Vaaler in \cite{HV}.  Another minimization problem with monotonicity restrictions was considered in \cite{CL2023} by Carneiro and Littmann, in the setting of one-sided majorants for the signum function.

In the following theorem, we present some qualitative and quantitative information about this problem. 	
\begin{theorem} \label{5_06pm} The following statements about the monotone-delta problem hold:
		\begin{itemize}
			\item [(a)] There exists an even function $F\in \mathcal{F}_1$ with $F(0)=1$ that extremizes \eqref{6_53pm}.
			\item [(b)] All the zeros of any even extremizer $F$ lie in the set 
			$
			S=\{z\in\C: |\re{z}|>|\im{z}|>0\}.
			$
			\item [(c)] The constant $\mathcal{A}_1$ satisfies
			$
			1.2750<\mathcal{A}_1<1.27714.
			$
		\end{itemize}
	\end{theorem}
	
Part (a) follows from standard compactness arguments. For part (b) we will show that zeros outside of $\mathcal{S}$ will either force a function to be zero, by analytic continuation and the constraints in the class $\mathcal{F}_1$, or one can carefully remove said zero and arrive at a contradiction.  Part (c) of Theorem \ref{5_06pm} is constructive, and although our lower bound only coincides up the two first digits, we conjecture that the upper bound in part (c) is sharp, at least up the first four significant digits of $\mathcal{A}_1$ shown above. As evidence, we exhibit  concrete examples for which the value 1.2771\ldots is attained. To estimate $\mathcal{A}_1$, we first reformulate the monotone-delta problem (see Lemma \ref{7_25pm} below) to the one of determining the infimum 
 \begin{equation}\label{eq:A12}
    \mathcal{A}_1=	\inf_{\substack{h\in \mathcal{F}_2\\h\not\equiv0}}\dfrac{2\displaystyle\int_{-\infty}^{\infty}|x|^2|h(x)|^2\,\dx}{\displaystyle\int_{-\infty}^
		\infty|x||h(x)|^2\,\dx},
\end{equation}
where the family $\mathcal{F}_2$ consists of entire functions $h:\C\to\C$ of exponential type at most $\pi$ such that $xh\in L^{2}(\R)$ and $|h(x)|=|h(-x)|$. We then introduce an $L^{2}$ approach to generate upper and lower bounds that  converge to $\mathcal{A}_1$ (see Theorem \ref{thm:lower_bounds}), and we use this approach to computationally obtain high precision numerical bounds, with rigorous computations in Ball arithmetic -- see Section \ref{sec:lower_bounds}.
Furthermore, we also find an explicit, relatively simple example of a function $h_0\in \mathcal{F}_2$ (see \eqref{eq:hPold2}). For this $h_0$, we compute explicitly the quotient in (\ref{eq:A12}), which turns out to be $1.2771\ldots$. Despite that $h_0$ is not the extremal function for (\ref{eq:A12}), our conjecture is that the value $1.2771\ldots$ is so close to the infimum $\mathcal{A}_1$, that they differ only in the decimal digits after the fourth one. Section \ref{sec:lower_bounds} is dedicated to prove part (c). See also \cite{GorM, HB} for works involving similar problems with computational approaches to solutions.

\smallskip

The monotone-delta problem has also been considered in $\mathbb{R}^d$, for $d\geq 2$. In \cite{Carneiroteam}, using techniques from the theory of de Branges spaces, the authors found the exact solution of the monotone-delta problem when $d$ is even. Nonetheless, the authors state that the case when $d$ is odd seems more subtle and remains open. 

\smallskip

Despite that Lemma \ref{7_25pm} below provides an integral representation of any function in $\mathcal{F}_1$, the first interesting explicit example of a function in this class we constructed was based on the classical method of  Sonin, which was itself invented with the intention to obtain information about the monotonicity of the successive relative minima and maxima of certain oscillatory solutions of ordinary differential equations (see \cite[Section 7.31]{G}). If $g:\C \to\C$ is a real entire function in $PW^2$ and satisfies a second-order differential equation of the form $y''+(B/x)\,y'+Cy=0$, with constants $B,C>0$, Sonin's method suggests to construct the function
	\begin{align}  \label{11_38pm}
		f(z)=(g(z))^2+\dfrac{(g'(z))^2}{C}.
	\end{align}
By the Plancherel-P\'olya theorem, since $g\in PW^2$ we have that $g'\in PW^2$, and therefore $f\in\mathcal{F}_1$. Moreover $f(x)$ is a ``lid" of $g^2(x)$ in the sense that $f(x)\geq g^2(x)$ for every $x\in \mathbb{R}$ and $f$ interpolates $g^2$ and possesses inflection points at its local maxima.  
Figure \ref{fejerCover} shows 
Fej\'er's kernel $K(x)$ and its lid $f(x)$.
	\begin{figure}[ht] 
		\centering
	\includegraphics[width=8cm, height=4cm]{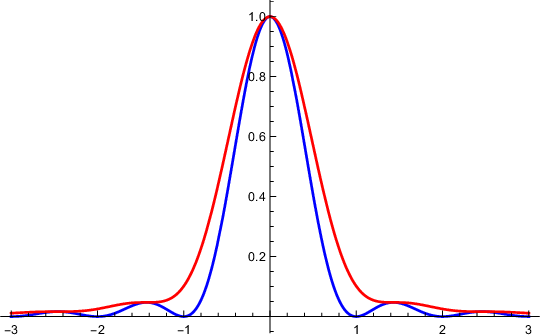}
	\caption{The {\color{blue}{Fej\'er kernel} $K(x)$} defined in \eqref{4_8_1:38pm} and its {\color{red}{lid $f(x)$}}.}
 \label{fejerCover}
	\end{figure}

\subsection{The one-delta problem with derivatives} The function in (\ref{11_38pm}) appears in a classical inequality for entire functions. Duffin and Schaeffer  \cite[p. 239]{DSB} proved that if a real entire function $g:\C\to\C$ of exponential type at most $\pi$ is such that $|g(x)|\leq 1$ for all $x\in\R$, then
		\begin{equation*} 
		|g(x)|^2+\dfrac{|g'(x)|^2}{\pi^2}\leq 1,\,\,\,\,\,\, \mbox{for all} \,\,x\in\R.
	\end{equation*} 
Inspired by this inequality, we prove that specific sums of the $L^2-$norms of a function $g \in PW^2$, normalized by $g(0)=1$, and its consecutive derivatives, are bounded from below. Our result may be considered a variation of the one-delta problem where one wishes to minimize sums of $L^2-$norms of an entire function and of its derivatives, and reads as follows:

	\begin{theorem} \label{2_57pm} 
	Let $N$ be a nonnegative integer and the real polynomial
	$$
	\mathcal{P}(x) = \displaystyle\sum_{n=0}^{N}a_n x^{n}
	$$ 
	be positive for every $x \in [0,1]$. 
	Then the inequality
	\begin{align} \label{10_02pm} \left(\int_{0}^{1}\dfrac{1}{\mathcal{P}(t^2)}\dt\right)^{-1}\leq 	\int_{-\infty}^{\infty}\displaystyle\sum_{n=0}^N\dfrac{a_n}{\pi^{2n}}\big|g^{(n)}(x)\big|^2\,\dx
	\end{align} 
	holds for every $g\in PW^2$ which obeys the normalization $g(0)=1$. Moreover, equality in \eqref{10_02pm} is attained if and only if 
	\begin{align} \label{10_44pm}
		g(z)=\left(\int_{0}^{1}\dfrac{1}{\mathcal{P}(t^2)}\dt\right)^{-1}\int_{0}^{1}\dfrac{\cos{(\pi zt)}}{\mathcal{P}(t^2)}\,\dt.
	\end{align} 
\end{theorem}
Note that when $N=0$ and $a_0=1$, we recover the inequality \eqref{9_59pm}, which once again shows that the latter is a natural result in the spirit of the one-delta problem. Moreover, choosing the polynomial $\mathcal{P}(x)=1+a\pi^2x$, we obtain the following corollary.

\begin{corollary}
	Fix $a>0$. Then the inequality
	\begin{align}  \label{10_57pm}
		\dfrac{\pi\sqrt{a}}{\arctan(\pi\sqrt{a})} \leq \int_{-\infty}^{\infty}\left(|g(x)|^2 +a\left|g^{\prime}(x)\right|^2\right)\,\dx,
	\end{align}
	holds for every $g\in PW^2$ with $g(0)=1$ and the unique extremal function for which \eqref{10_57pm} reduces to an equality is 
	$$
	g(z)=\dfrac{\pi\sqrt{a}}{\arctan(\pi\sqrt{a})}\int_{0}^{1}\dfrac{\cos(\pi zt)}{1+a\pi^2t^2}\,\dt.
	$$ 
\end{corollary}Observe that for $a=1/\pi^2$ (\ref{10_57pm}) reduces to the following estimate:
$$
\int_{-\infty}^{\infty}\left(|g(x)|^2 + \frac{\left|g^{\prime}(x)\right|^2}{\pi^2}\right)\,\!\dx \geq \frac{4}{\pi},\ \ \  g\in PW^2,\ \  g(0)=1.
$$
Different choices of the polynomial $\mathcal{P}(x)$ allow us to obtain other interesting inequalities. 
\begin{corollary}
	Fix $0<a<{1}/{\pi^2}$. Then
	\begin{align} \label{3_48pm}
		a\int_{-\infty}^{\infty}\big|g^{\prime}(x)\big|^2\dx+ \bigg(\dfrac{1}{2\pi\sqrt{a}}\,\log\bigg(\dfrac{1+\sqrt{a}\pi}{1-\sqrt{a}\pi}\bigg)\bigg)^{-1}\leq \int_{-\infty}^{\infty}|g(x)|^2\,\dx.
	\end{align}
	for every $g\in PW^2$ which obeys $g(0)=1$.
\end{corollary}

In particular, letting $a\to {1}/{\pi^2}$ in \eqref{3_48pm} we obtain
\begin{align*} 
	\int_{-\infty}^{\infty}\big|g^{\prime}(x)\big|^2\,\dx\leq \pi^2\int_{-\infty}^{\infty}|g(x)|^2\,\dx,\ \ g\in PW^2,\ \ g(0)=1,
\end{align*}
which is exactly the $L^2-$version of the classical Bernstein inequality that holds for every $L^p(\R)$, $p\geq 1$ (see \cite[Theorem 11.3.3]{Boas}).

Observe that the Bernstein inequality follows from Theorem \ref{2_57pm}  if we set $\mathcal{P}(t)=1+\varepsilon -t$ and let $\varepsilon \to 0$. Applying the same reasoning with $\mathcal{P}(t) = (1+\varepsilon -t)^N$, we obtain: 
\begin{corollary}
Let $N$ be a nonnegative integer. Then the inequality 
$$
\sum_{k=0}^N \, \frac{(-1)^k}{\sigma^{2k}}\, {N \choose k}\int_{-\infty}^\infty \big| f^{(k)}(x)\big|^2\,\dx\geq 0
$$
holds for every  function of exponential type at most $\sigma$ such that $f \in L^2(\mathbb{R})$. 
   In particular, for $N=2$, 
$$
\int_{-\infty}^\infty \big| f^{\prime}(x)\big|^2\,\dx \leq \frac{1}{2} \left( \sigma^2 \int_{-\infty}^\infty \big| f(x)\big|^2\,\dx + \frac{1}{\sigma^{2}} \int_{-\infty}^\infty \big| f^{\prime\prime}(x)\big|^2\,\dx \right).
$$
\end{corollary}

The latter is a curious result that resembles some classical ones, due to Landau and Hardy and Littlewood. In 1913, Landau \cite{Landau} proved that if $f$ is a real function, $f \in C^2(\mathbb{R})$, and the inequalities $\| f \|_\infty \leq 1$ and $\| f^{\prime\prime} \|_\infty \leq 1$ for the uniform norms of $f$ and $f^{\prime\prime}$ on the real line hold, so does $\| f^{\prime} \|_\infty \leq \sqrt{2}$.   

Hardy and Littlewood \cite[Theorem 6]{HL} proved that, if $y$ and $y''$ are in $L^2[0,\infty)$, then 
$$
\left( \int_{0}^\infty [ y'(x) ]^2 \,\dx \right)^2 \leq 4 \int_{0}^\infty [ y(x) ]^2\,\dx \int_{0}^\infty [ y''(x) ]^2\,\dx.
$$
Moreover, the constant $4$ is the best possible. The equality is attained if and only if $y(x) = c\, Y(ax)$, where $c$ and $a$ are real constants and 
$$
Y(x) = e^{-x/2} \sin \left( \frac{\sqrt{3}}{2}\, x - \frac{\pi}{3} \right).
$$
Theorem 7 in \cite{HL} states that, under the same requirements, the inequality 
$$
\int_{0}^\infty ( y^2(x) + [y''(x)]^2 - [y'(x)]^2 )\,\dx \geq 0
$$
holds with equality as before, but with $a=1$.

\section{Proof of Theorem \ref{5_06pm}: Qualitative aspects} \label{Section Proof}
	
	For $f\in L^1(\R)$, we normalize the Fourier transform $\widehat{f}$ of $f$ as 
	$$
	\widehat{f}(\xi)=\int_{-\infty}^\infty f(x)e^{-2\pi ix\xi}\,\dx.
	$$

\subsection{Proof of a)}  Replacing $f(x)$ by $(f(x)+f(-x))/2$, we see that we may restrict our search for the infimum \eqref{6_53pm} to the even functions in $\mathcal{F}_1$.  Consider an extremizing sequence $\{f_n\}_{n\geq 1}\subset\mathcal{F}_1$ such that $f_n$ is even, $f_n(0)=\|f_n\|_\infty=1$, and $\|f_n\|_1\to\mathcal{A}_1$.  It follows from \cite[Theorem 3.3.6]{Niko}, by passing to a  subsequence if necessary, that there is $F:\C\to\C$  of exponential type at most $2\pi$ such that $F \in L^1(\R)$ and
\begin{align*} 
	f_n(x)  \to F(x)\ \ \ \mathrm{as}\ \  n\to\infty
\end{align*}
uniformly on any compact set of $\C$. Therefore, $F$ is even, $F\in \mathcal{F}_1$ and $F(0)=1$.  Fatou's lemma implies that $\|F\|_1\leq \mathcal{A}_1$, and by the definition of $\mathcal{A}_1$ as an infimum we conclude that $F$ is an extremizer for \eqref{6_53pm}. 



\subsection{Proof of b)} Let $F$ be an even extremizer with $F(0)=1$. Clearly, it has no real zeros. Indeed, since $F$ is real, nonnegative and decreasing on the positive real axis, if it vanishes at $x_0>0$, it does for all $x>x_0$ which is impossible because $F$ is entire and $F(0)=1$. Since $F$ is even, also it does not vanish at a negative $x_0$.  
Therefore, all the zeros of $F$ satisfy $|\im{z}|>0$. Now, assume that $F$ has a zero at $z=ib$, for $b\in\R$. Since $F$ is real-valued, it also has $z=-ib$ as a zero. Consider the entire function
$$
G(z)=\dfrac{b^2F(z)}{z^2+b^2}.
$$
Note that $G(0)=1$ and $G\in \mathcal{F}_1$. Since
$$
\int_{-\infty}^{\infty}G(x)\,\dx<\int_{-\infty}^{\infty}F(x)\,\dx
$$
we get a contradiction. Therefore, all the zeros of $F$ satisfy $|\re{z}|>0$. Now, assume that $z=a+ib$ is a zero of $F$ with $|b|\geq |a|>0$. Since $F$ is real-valued and even, we have that $z=a-ib$, $z=-a+ib$, and $z=-a-ib$ are also zeros. Note that all these zeros are different. Then, the entire function

$$
H(z)=\dfrac{(a^2+b^2)^2F(z)}{\left((z-a)^2+b^2\right)\left((z+a)^2+b^2\right)}
$$
is in $\mathcal{F}_1$, and using that $|b|\geq |a|$, it is easy to see that
$$
\int_{-\infty}^{\infty} H(x)\,\dx < \int_{-\infty}^{\infty} F(x)\,\dx 
$$
which gives a contradiction. We conclude that $|b|<|a|$.

\section{Proof of Theorem \ref{5_06pm}: quantitative aspects}\label{sec:lower_bounds}

\subsection{Representation lemma} The following lemma gives a representation for any even function in $\mathcal{F}_1$. 
\begin{lemma} \label{7_25pm} If $f\in\mathcal{F}_1$ is even, then it can be represented in $\R$ in the form
	\begin{align} \label{9_41pm}
		f(x)=\int_{-\infty}^{x}-t\,|h(t)|^2\,\dt,
	\end{align}
	where $h:\C\to\C$ is an entire function of exponential type at most $\pi$ such that $|h(x)|=|h(-x)|$ for all $x\in\R$, and
	$xh\in L^2(\R)$. Conversely, if $f$ is a function of the form \eqref{9_41pm}, then it has an analytic extension to $\C$ which is an even function in $\mathcal{F}_1$.
\end{lemma}
\begin{proof} Let $f\in\mathcal{F}_1$ be even.  Integration by parts yields 
	\begin{align} \label{6_36pm}
		\int_{0}^{x}f(t)\,\dt = xf(x)+\int_{0}^{x}t\, |f'(t)|\,\dt.
	\end{align} 
 Since the integrals on both sides of \eqref{6_36pm} are increasing functions of $x$, and $f\in L^1(\R)$, when $x\to+\infty$  we can conclude that $\displaystyle\lim_{x\to\infty} xf(x)$ exists (similarly $\displaystyle\lim_{x\to-\infty} xf(x)$ exists). The fact that $f \in L^1(\R)$ forces these limits to be zero and one can also conclude that $x\,f'(x)\in L^1(\R)$. By the Plancherel-P\'olya theorem, $f'(z)$ has exponential type $2\pi$ and so does $-zf'(z)$. From the Krein decomposition theorem \cite[p. 154]{Krein}, it follows that $-zf'(z)=g(z)g^*(z)$ for some $g\in PW^2$. Moreover, since $f$ attains its maximum at $x=0$, then $f'(0)=0$. Defining $h(z)=g(z)/z$, we rewrite the latter in the form
	\begin{align}  \label{6_33pm}
		-zf'(z)=z^2h(z)h^*(z),
	\end{align}
	where $h$ is an entire function of exponential type at most $\pi$ and $xh\in L^2(\R)$. Since $f'$ is odd, then $|h(x)|=|h(-x)|$ for $x\in\R$. Finally, integrating \eqref{6_33pm} appropriately, we arrive at \eqref{9_41pm}. Conversely, assume the representation \eqref{9_41pm}. Note that $f$ has an analytic extension on $\C$ (also denoted by $f$) of the form
	\begin{align*}
		f(z)= \int_{-\infty}^0-t|h(t)|^2\,\dt+\int_{[0,z]}-s\,h(s)h^*(s)\,\ds,
	\end{align*} 
where $[0,z]$ denotes the straight segment connecting $0$ and $z$.
	Since $h$ is an entire function of exponential type at most $\pi$, $f$ is an entire function of exponential type at most $2\pi$. From \eqref{9_41pm} it follows that $\displaystyle\lim_{x\to-\infty}f(x)=0$, and using the fact that $|h(x)|=|h(-x)|$, we conclude that $f$ is also even and $\displaystyle\lim_{x\to\infty}f(x)=0$. On the other hand, differentiating \eqref{9_41pm} we derive
\begin{align}  \label{8_51pm}
	f'(x)=-x\, |h(x)|^2\ \ \ \mathrm{for}\ \ x\in\R,
\end{align}
	 which implies that $f$ is radially decreasing and $f(x)\geq 0$. Moreover, \eqref{8_51pm} and $xh \in L^1(\R)$ imply $\displaystyle\lim_{x\to  \pm\infty}xf(x)=0$. Integration by parts shows that
		\begin{align*}
		\int_{-\infty}^{\infty}f(x)\,\dx =\int_{-\infty}^{\infty}x^2|h(x)|^2\,\dx,
	\end{align*} 
which yields $f\in L^1(\R)$. \end{proof}

From Lemma \ref{7_25pm}, we can reformulate the monotone-delta problem as the one to determine 
\begin{equation} \label{eq:A123}
    \mathcal{A}_1=	\inf_{\substack{h\in \mathcal{F}_2\\h\not\equiv0}}\dfrac{2\displaystyle\int_{-\infty}^{\infty}|x|^2|h(x)|^2\,\dx}{\displaystyle\int_{-\infty}^
		\infty|x||h(x)|^2\,\dx},
\end{equation} 
where the family $\mathcal{F}_2$ consists of those entire functions $h:\C\to\C$ of exponential type at most $\pi$ such that $xh\in L^{2}(\R)$ and $|h(x)|=|h(-x)|$.

\subsection{An $L^2-$computational approach}\label{sec:L2}
A natural approach for constructing functions in $\mathcal{F}_2$ (and therefore in $\mathcal{F}_1$), and computationally solving \eqref{eq:A123}, starts by finding an orthonormal system for the space $L^2(\R,x^2\dx)$. 
Note that $\mathcal{F}_2$ is a Hilbert space with the inner product 
\begin{equation*}
 \langle f,g\rangle_{\mathcal{F}_2} = \langle xf, xg \rangle_{L^2(\R)},
\end{equation*}
and norm
\begin{align} \label{normF2}
\|f\|_{\mathcal{F}_2}=\left(\int_{-\infty}^{\infty}|x|^2|f(x)|^2\,\dx\right)^{\!\frac{1}{2}}.
\end{align}
For positive integers $k$, we define the even functions
\begin{equation}\label{eq:hk_def}
    h_k(x)=\frac{4\sqrt{2}}{\pi}\cdot \frac{ \cos \pi x}{ (2k-1)^2 - 4x^2},
\end{equation}
and note that $h_k\in \mathcal{F}_2$ for all positive integers $k$. Gorbachev \cite{Gor2005} previously considered this family of functions to obtain fine numerical estimates for other Fourier extremal problems, and it has also been used in \cite{CQH} for similar purposes in related extremal problems introduced by Carneiro, Milinovich, and Soundararajan \cite{CMS}. Regarding this system, we can say the following:
\begin{proposition}\label{prop:complete} The family $(h_k)_{\substack{  k\ge 1}}$ is a complete orthonormal system in the closed subspace $\{h\in \mathcal{F}_2: \ h\text{ is even.} \}.$
\end{proposition}

\begin{proof}

Note that, if $h\in \mathcal{F}_2$ is even, then $xh\in L^2(\R)$ is odd. Furthermore, we have that 
\begin{equation}\label{eq:sk}
 \widehat{(xh_k)}(t)= i(-1)^{k}\sqrt{2} \sin(\pi (2k-1) t )\chi_I(t) =: s_k(t),
\end{equation}
where $I=[-1/2,1/2]$ and $\chi_{I}$ denotes the characteristic function of the interval $I$. To see this, since $s_k\in L^1(\R)\cap L^2(\R)$, we may compute $\widehat s_k$ in a straightforward manner to verify that $\widehat s_k(x)=-xh_k(x)$, and then we conclude \eqref{eq:sk} by Fourier inversion in $L^2(\R)$.
Now consider the operator 
\begin{equation*}
 T:\mathcal{F}_2\to L^2(I)
\end{equation*}
defined by $Th(t):= \widehat {(xh)}(t)e^{\pi i t }$. 
By Plancherel's theorem and the Paley-Wiener theorem, $T$ is a linear isometry, that is, $\langle f, g \rangle_{\mathcal{F}_2} = \langle Tf, Tg \rangle_{L^2(I)}$. Therefore, for positive integers $k$ and $j$, we find that 
\begin{equation*}
 \langle h_k, h_j \rangle_{\mathcal{F}_2} = \langle s_k, s_j \rangle_{L^2(I)}= \delta_{kj},
\end{equation*}
where $\delta_{kj}=1$ if $k=j$, and $0$ otherwise. Here, to compute the inner product over $L^2(I)$, we may apply the identity $2\sin(\pi (2k-1) t )\sin(\pi (2j-1) t ) = \cos(2\pi (k-j) t )-\cos(2\pi (k+j-1) t )$. This shows that $h_k$ is orthonormal.

We now show that it is complete. Let $h\in\mathcal{F}_2$ be even, such that $\langle h, h_k\rangle_{\mathcal{F}_2}=0$ for all positive integers $k$. We must show that $h\equiv 0$. First, denote $H(t)=\widehat{(xh)}(t)$, and note that, by Plancherel's theorem and \eqref{eq:sk}, the condition $\langle h, h_k\rangle_{\mathcal{F}_2}=0$ implies that 
\begin{equation}\label{eq:Horthog}
\int_I H(t)\sin(\pi(2j-1)t)\,\d t =0
\end{equation}
for all positive integers $j$. Actually, since $\sin (-x)=-\sin x$, \eqref{eq:Horthog} holds for all integers $j$. 

Now, since $T$ is an isometry into $L^2(I)$, by the theory of Fourier series on $L^2(I)$, it is enough to show that $\langle Th, e_j\rangle_{L^2(I)}=0$ for all integers $j$, where $e_j(t)=e^{2\pi i jt}$. In fact, for an integer $j$, we have
\begin{equation*}
\begin{split}
  \langle Th, e_j\rangle_{L^2(I)} &= 
  \int_I H(t)e^{\pi i t}e^{-2\pi i t j }\,\d t\\
  &= \int_I H(t)\cos(\pi(2j-1)t)\,\d t - i\int_I H(t)\sin(\pi(2j-1)t)\,\d t.
\end{split}
\end{equation*}
The first integral in the last line is 0 since $H$ is odd, and the second integral is 0 by \eqref{eq:Horthog}. Therefore, $\langle Th, e_j\rangle_{L^2(I)}=0$ for all integers $j$, and then $Th\equiv 0$ and $h\equiv 0$, as desired.
\end{proof}

More general orthogonality results can also be obtained from the theory of de Branges spaces; see \cite[Section 2.2]{Carneiroteam} and the references therein. See also \cite[Theorem 4.2]{Gor2019} for a similar general orthogonality result.

\subsection{Proof of c): Generating bounds}\label{sec:upper_bounds}
Once we have a complete orthonormal system, we proceed to obtain numerical examples as follows. For a positive integer $d$, let $\mathcal{F}_{2,d}=\,$span $ \{h_k : 1\le k\le d\}\subset \mathcal{F}_2$. Let $Q\in\R^{d\times d}$ be the matrix defined by 
\begin{equation}\label{eq:lambda}
    Q_{kj} = \int_{0}^\infty x\,h_k(x)h_j(x)\,\d x.
\end{equation}
Then, since $h_k$ are orthonormal, one can see that the reciprocal of the infimum in \eqref{eq:A123}, when taken over the space $\mathcal{F}_{2,d}$, satisfies 
\begin{equation*}
    |\lambda_d| = \max_{\substack{h\in \mathcal{F}_{2,d}\\h\not\equiv0}} \dfrac{\displaystyle\int_{0}^
		\infty x\, |h(x)|^2\,\dx}{\displaystyle\int_{-\infty}^{\infty}|x|^2|h(x)|^2\,\dx},
\end{equation*}
where $\lambda_d$ is the largest eigenvalue (in absolute value) of $Q$, and the maximum is attained when 
\begin{equation}\label{eq:hL2}
h=\mathbf{a}\cdot (h_1,h_2,\ldots, h_{d}), 
\end{equation}
for $\mathbf{a}\in \R^{d}$ an eigenvector of $Q$ associated to $\lambda_d$. In particular, we have that $\mathcal{A}_1\le |\lambda_d|^{-1}$. 
Moreover, we now proceed to prove that as $d\to\infty$, we have $|\lambda_d|^{-1}\to \mathcal{A}_1$. Furthermore, we are able to explicitly estimate the speed of convergence, yielding lower bounds that also converge to $\mathcal{A}_1$ (see Theorem \ref{thm:lower_bounds} below). 
With that goal in mind, we proceed to study the coefficients $Q_{kj}$ defined in \eqref{eq:lambda}.  Our first observation is that they can be made more explicit. When $k > j$, we expand in partial fractions and use the  trigonometric identity $2\cos^2(x)=1+\cos(2x)$ to see that
\begin{align*}
    &(k-j)(k+j-1)Q_{kj}=\int_{0}^\infty \frac{32(k-j)(k+j-1)x\cos^2(\pi x)}{\pi^2(4x^2-(2k-1)^2)(4x^2-(2j-1)^2)}\dx \\
    &=\int_{0}^\infty \frac{1+\cos(2\pi x)}{\pi^2}\left(\frac{1}{2x-(2k-1)}+\frac{1}{2x+(2k-1)}-\frac{1}{2x-(2j-1)}-\frac{1}{2x+(2j-1)}\right)\dx.
\end{align*}
By carefully splitting the integral, changing variables by translations and dilations, and regrouping all the pieces, one obtains
\begin{equation*}
    Q_{kj} =\dfrac{-1}{\pi^2 \, (k-j) \,(k+j-1)} \displaystyle\int_{2j-1}^{2k-1} \frac{1-\cos(\pi x)}{x}\,\d x.
\end{equation*}
When $k=j$, a similar argument leads to the expression
\begin{align*}
    Q_{kk} &= \dfrac{2}{\pi^2(2k-1)}\int_0^{2k-1}\frac{1-\cos(\pi x)}{x^2}\,\d x= \dfrac{-4+2(2k-1)\pi \, \text{Si}(\pi (2k-1))}{\pi^2\, (2k-1)^2},
\end{align*}
  where $\mathrm{Si}(x)=\int_0^x \sin(t)/t\,\d t $ is the standard sine integral function. These expressions readily imply that
 \begin{align*}
      |Q_{kj}|\lesssim \frac{1}{d\,(k-j)},\,\,\,\,\, \mbox{and}\,\,\,\,\,
      |Q_{kk}|\lesssim \frac{1}{(2k-1)},\,\,\,\,\, \mathrm{when}\,\,d\leq j< k. 
  \end{align*}
These inequalities, with effective constants, lead to bounds that can be used to explicitly estimate the speed of convergence of $|\lambda_d|^{-1}$ to $\mathcal{A}_1$. By taking a particular value of $d$, we obtain the bounds stated in Theorem \ref{5_06pm}.  Before we state our general bounds, we briefly remark that, from the aforementioned work \cite{Carneiroteam} (see Theorem 2 therein), we may restrict the search for the infimum in \eqref{eq:A12} to \textit{even} functions $h\in \mathcal{F}_2$.

\begin{theorem}\label{thm:lower_bounds}
	Let $h\in\mathcal{F}_2$ be even and not identically 0. Define $h_k$ as in \eqref{eq:hk_def}, and for a positive integer $d$, let $\lambda_d$ be the maximum eigenvalue (in absolute value) of the matrix $Q$ defined in \eqref{eq:lambda}. Then, for $d\ge 1000$, 
	\begin{equation*}
		\displaystyle\int_{-\infty}^
		\infty |x|\,|h(x)|^2\,\dx 
		< 2 \left(
		{\displaystyle\int_{-\infty}^{\infty}|x|^2|h(x)|^2\,\dx}
		\right)
		\left( |\lambda_d| + \frac{\sqrt{\log d}}{d} + \frac{1}{d} \right).
	\end{equation*}
 In particular, for any $d\ge 1000$, we have 
	\[ \left( |\lambda_d| + \frac{\sqrt{\log d}}{d} + \frac{1}{d} \right)^{-1} < \mathcal{A}_1 \le |\lambda_d|^{-1}.\]
	When $d=3010$, one obtains\footnote{These numerical computations were rigorously verified by using ball arithmetic with the Arb library \cite{arb}. 
	} 
	\begin{equation*}
		1.2750 <\mathcal{A}_1<1.27714.
	\end{equation*}
\end{theorem}
\begin{proof}
 Write $h=\sum_{k=1}^\infty a_kh_k$, with $a_k\in\C$. Denote $\mathbf{a}=(a_k)_{k\ge 1}$ and $\|\mathbf{a}\|_2:=\sqrt{\sum_{k\ge 1}|a_k|^2 }$, so that $\|h\|_{\mathcal{F}_2}=\|\mathbf{a}\|_2$ (see \eqref{normF2}). Now, consider the functional $J:\mathcal{F}_2\backslash\{0\}\to \R$ given by \[J(h):=\frac{\displaystyle\int_{0}^
	\infty x\,|h(x)|^2\,\dx } 
{\displaystyle\int_{-\infty}^{\infty}|x|^2|h(x)|^2\,\dx}. \]
 First note that $J$ is continuous on $\{h\in\mathcal{F}_2:\, h\not\equiv 0\}$. Indeed, from the classical one-delta problem, we have the trivial lower bound $\mathcal{A}_1\ge 1.$ For $f_1,$ $f_2\in \mathcal{F}_2$, the triangle inequality 
 yields 
 \[
 \left|\sqrt{\int_{0}^\infty|x|\,|f_1(x)|^2\,\d x}-
 \sqrt{\int_{0}^\infty|x|\,|f_2(x)|^2\,\d x}
 \right| \le 
 \sqrt{\int_{0}^\infty|x|\,|f_1(x)-f_2(x)|^2\,\d x} \le \|f_1-f_2\|_{\mathcal{F}_2},
 \]
 which implies the desired continuity. Here, we used that $\mathcal{A}_1\le 1$ in the last inequality.
 Now, fix $d\ge1000$, and for a parameter $M\geq d$, consider $H_M=\sum_{k=1}^M a_kh_k$. One has
 
\begin{align}\label{eq:Jh}
    J(H_M)\cdot \left(|a_1|^2+|a_2|^2+\ldots+|a_M|^2\right) &= \sum_{k,j=1}^d a_k \ov{a_j}\, Q_{kj} 
    + 2\, \re \sum_{\substack{ d < k\leq M 
    \\ 1\le j\le d } }
    a_k \ov{a_j}\, Q_{kj} 
    + \sum_{ k,j=d+1}^M a_k \ov{a_j}\, Q_{kj} \nonumber \\
    &= J(H_d)\cdot \left(|a_1|^2+|a_2|^2+\ldots+|a_d|^2\right) + J_1 + J_2
\end{align}
Here, $Q_{kj}$ is defined as in equation \eqref{eq:lambda}. By definition of $\lambda_d$, we have $|J(H_d)|\le  |\lambda_d|$. We now estimate $|J_1|$ and $|J_2|$.  One can verify that  $Q_{kk}>0$ and that $Q_{kj}<0$ for $k\neq j$. 
  To obtain upper bounds for $|Q_{kj}|$, note that, since the local maxima of Si$(x)$ form a decreasing sequence, we have $0<\text{Si}(\pi(2k-1))\leq\text{Si}(2001\pi )$ for all $k>1000$. Additionally, since $|\cos(\pi x)|\le 1$, we have that, for $1\le j < k$:
  \[
  0<\int_{2j-1}^{2k-1} \frac{1-\cos(\pi x)}{x}\,\d x \le  \int_{2j-1}^{2k-1} \frac{2}{x}\,\d x = 2\log\left(1+\frac{2(k-j)}{2j-1}\right).
  \]
This yields the inequalities 
  \begin{align*}
      |Q_{kj}|\le \frac{2\log\left(1+\tfrac{2(k-j)}{2j-1}\right)}{\pi^2\,(k-j)\,(k+j-1)},\,\,\,\,\,\,\,\,\,\,\, \mbox{and}\,\,\,\,\,\,\,\,\,\,\,\,
      |Q_{kk}|\le \frac{2\, \text{Si} (2001\pi)}{\pi(2k-1)} 
  \end{align*}
for $1\leq j<k$, and for $k>1000$ respectively. To estimate $|J_1|$ we use the triangle inequality, extend the sum over $k$ to infinity and apply the Cauchy-Schwarz inequality in both variables, which yields
  \begin{align}\label{eq:J1}
      |J_1| &\le \frac{4}{\pi^2}\sum_{\substack{d < k 
    \\ 1\le j\le d } }
     \, \frac{\log\left(1+\tfrac{2(k-j)}{2j-1}\right)}{(k-j)(k+j-1) }|a_k a_j|\nonumber \\
    &\le \frac{4}{\pi^2} \left(\sum_{\substack{d < k 
    \\ 1\le j\le d } }
    \, \frac{\log^2\left(1+\tfrac{2(k-j)}{2j-1}\right)}{(k-j)^2(k+j-1)^2 }\right)^{\!\!\frac{1}{2}}\left(\sum_{k,j=1}^\infty|a_k|^2|a_j|^2\right)^{\!\!\frac{1}{2}}\nonumber\\
    &\le\frac{4}{\pi^2} \left(\sum_{j=1}^d
    \frac{1}{(j+d)^2}\,\sum_{k=d+1}^\infty \frac{\log^2\left(1+\tfrac{2(k-j)}{2j-1}\right)}{(k-j)^2 }\right)^{\!\!\frac{1}{2}}\|\mathbf{a}\|_2^2.
  \end{align}
Since $k-j\geq 1$, by applying estimates \eqref{eq:elem_2} and \eqref{eq:elem_1} with $n=d$ and $t={2}/{(2j-1)}$ 
 in \eqref{eq:J1}, we obtain
\begin{align*}
    |J_1|&\leq \frac{4}{\pi^2}\left(\sum_{j=1}^d
    \frac{4}{(j+d)^2(2j-1)^2}+\frac{\pi^2}{3}\sum_{j=1}^d
    \frac{2}{(j+d)^2(2j-1)}\right)^{\!\!\frac{1}{2}}\|\mathbf{a}\|_2^2\nonumber \\
    &\leq \frac{4}{\pi^2d}\left(\sum_{j=1}^\infty
    \frac{4}{(2j-1)^2}+\frac{\pi^2}{3}\sum_{j=1}^d
    \frac{2}{(2j-1)}\right)^{\!\!\frac{1}{2}}\|\mathbf{a}\|_2^2 \nonumber \\
    &\leq \frac{4}{\pi^2d}\left(\frac{\pi^2}{2}+\frac{\pi^2}{2}\log d\right)^{\!\!\frac{1}{2}}\|\mathbf{a}\|_2^2.
\end{align*}
In particular, when $d\geq 1000$,
\begin{align}
    |J_1|\leq \frac{2\sqrt{2}}{\pi}\left(1+\frac{1}{\log 1000}\right)^{\!\!\frac{1}{2}}\frac{\sqrt{\log d}}{d}\|\mathbf{a}\|_2^2 < \frac{\sqrt{\log d}}{d}\|\mathbf{a}\|^2. \label{eq:alpha}
\end{align}
To estimate $|J_2|$, we first separate the diagonal term $k=j$. On the other terms, we use the fact that $x\mapsto {\log(1+ax)}/{x}$ is decreasing when $x\in [0,\infty)$ for any $a>0$, and the Cauchy-Schwarz inequality, to obtain
  \begin{align}\label{eq:J2}
      |J_2| &\le \frac{2 \, \text{Si}(2001\pi)}{\pi (2d+1)}\cdot \sum_{k>d}|a_k|^2 + \frac{4}{\pi^2}\sum_{\substack{k> j \\ k,\,j>d}} |a_k| |a_j|\frac{\log\left(1+\tfrac{2}{2j-1}\right)}{k+j-1}
      \nonumber \\
      &\le 
      \frac{ \,\text{Si}(2001\pi)}{\pi\, d}\|\mathbf{a}\|_2^2 + \frac{8}{\pi^2}\left( \sum_{\substack{k> j \\ k,\,j>d}}\frac{1}{(k+j-1)^2(2j-1)^2}\right)^{\!\!\frac{1}{2}}\|\mathbf{a}\|_2^2 \nonumber
  \end{align}
  By applying \eqref{eq:elem_3} with $n=2j$ followed by \eqref{eq:elem_4} with $n=d+1$, one arrives at
\begin{equation}\label{eq:beta}
    |J_2| \leq \frac{ \,\text{Si}(2001\pi)}{\pi\, d}\|\mathbf{a}\|_2^2 + \frac{8}{\pi^2}\left(1+\frac{1}{1000}\right)^{\!\!\frac{1}{2}}\left( \sum_{j=d+1}^\infty\frac{1}{2j(2j-1)^2}\right)^{\!\!\frac{1}{2}}\|\mathbf{a}\|_2^2  < \frac{1}{d}\|\mathbf{a}\|^2.
\end{equation}
Since \eqref{eq:alpha} and \eqref{eq:beta} do not depend on $M$, applying these estimates on \eqref{eq:Jh}, sending $M\to\infty$, and using the continuity of $J$ concludes the proof. Finally, to obtain numerical bounds, we calculate the eigensystems numerically for $d\le 3010$, and find that $
			0.783002554179<\lambda_{3010}<0.783002554181$.
		We thereby obtain the numerical bounds in Theorem 8, and we highlight here our best, rigorous upper bound: 
		\[\mathcal{A}_1 < 1.277135042105.\]

\end{proof}


	
	\section{Some functions in $\mathcal{F}_1$} \label{Examples}
	
	\subsection{The lid function} In this subsection, we apply Sonin's method to construct a nice sequence of functions in $\mathcal{F}_1$. For any positive real numbers $B$ and $C$, consider the differential equation
	\begin{align} \label{18_7_17:38pm}
		y''+\dfrac{B}{x}\,y'+Cy=0. 
	\end{align} 
	Let $y=g$, $g:\R\to\R$ be a solution of the equation \eqref{18_7_17:38pm}. The \emph{lid} of $g^2$ is the function defined by
\begin{align}  \label{4_54pm}
	f(x)=(g(x))^2+\dfrac{(g'(x))^2}{C}.
\end{align}
	Note that $f(x)\geq 0$ for all $x\in\R$, and
	\begin{align*}
		f'(x)=-\dfrac{2B(g'(x))^2}{x\,C}.
	\end{align*} 
	This implies that $f$ is radially decreasing. Moreover, if we suppose that the solution $g$ has an analytic extension on $\C$ of exponential type at most $\pi$, and $g\in L^{2}(\R)$, we conclude that $f\in \mathcal{F}_1$. 
	
	Let us show some examples of lids. For $\alpha>0$, consider the Bessel function of the first kind of order $\alpha$, which is defined by 
	\begin{align*}  
		J_{\alpha}(z)=\displaystyle\sum_{\nu=0}^{\infty}\dfrac{(-1)^{\nu}(z/2)^{\alpha+2\nu}}{\nu!\,\Gamma(\nu+\alpha+1)}.
	\end{align*}
Let us remark some properties of the Bessel functions mentioned in \cite[Section 1.71]{G}.
It is known  (see \cite[Equation 1.71.3]{G}) that $J_\alpha$ satisfies the differential equation
\begin{equation} \label{4}
		y'' + x^{-1}y' + (1-\alpha^2x^{-2})y=0.
\end{equation} 
Now, define the function
$$
g_\alpha(z)=\dfrac{J_\alpha(\pi z)}{(\pi z)^\alpha}.
$$
A straightforward change of variables in \eqref{4} shows that $g_\alpha$ satisfies the differential equation
\begin{equation*} 
	y'' + \dfrac{2\alpha+1}{x}\,y' + \pi^2y=0.
\end{equation*} 
The function $g_\alpha$ is an even entire function of exponential type $\pi$. Moreover, using the decay of $J_\alpha$ (see \cite[Equations 1.71.10 and 1.71.11]{G} we see that $g_\alpha\in L^2(\R)$. Therefore, inserting $g_\alpha$ in \eqref{4_54pm} we actually construct the lid of $g_\alpha^2$, with $B=2\alpha+1$ and $C=\pi^2$. In the particular case $\alpha=1/2$ we known that 
$$
g_{1/2}(x)=\dfrac{\sin(\pi x)}{\pi x},
$$
and therefore 
\begin{align*}  
	f_{1/2}(x)=(g_{1/2}(x))^2+\dfrac{(g_{1/2}'(x))^2}{\pi^2}
\end{align*}
is the lid of $K(x)$. Straightforward calculations show that the Fourier transform of $f_{1/2}$ is
\begin{align} \label{21}
	\widehat{	f_{1/2}}(\xi)=\max\{1-|\xi|,0\}+\dfrac{1}{\pi^2}\widehat{(g'_{1/2})^2}(\xi).
\end{align}
Then the Fourier transform - convolution de Margan type law yields 
\begin{align*} 
	\widehat{(g_{1/2}')^2}(\xi) & = (\widehat{g_{1/2}'} \ast\widehat{g_{1/2}'})(\xi)=\left((2\pi i x\, \widehat{g_{1/2}})\ast(2\pi i x \,\widehat{g_{1/2}})\right)\!(\xi)= -4\pi^2\int_{-\infty}^\infty x\widehat{g_{1/2}}(x)(\xi-x)\widehat{g_{1/2}}(\xi-x)\,\dx,
\end{align*}
where we used the fact that $\widehat{g_{1/2}}(x)= \chi_{I}(x)$. This, together with \eqref{21}, implies
\begin{equation*}
	\widehat{f_{1/2}}\,(\xi) = \left\{
	\begin{array}{ll}
		\dfrac{2}{3}\,(1-|\xi|)^2(|\xi|+2),      & \mathrm{if\ } |\xi|\leq 1;\\
		0, & \mathrm{if\ } |\xi|> 1.
	\end{array}
	\right.
\end{equation*}
In particular, this example allows us to obtain the bound $\mathcal{A}_1\leq 1.333\ldots$. In fact, one can repeat the same argument for the function $f_\alpha$, for any $\alpha>0$. The Fourier transform of $f_\alpha$ can be computed using \cite[Equation 1.71.6]{G}. Finally, we minimize the ratio $\widehat{f_\alpha}(0)/f_\alpha(0)$ with respect to $\alpha$,  and obtain that it is attained for $\alpha_0=0.787\ldots$ and $\widehat{f_{\alpha_0}}(0)/f_{\alpha_0}(0)=1.284\ldots$. Hence $\mathcal{A}_1\leq 1.284\ldots$.
\subsection{Polynomial examples}\label{sec:proofNumeric} As mentioned in the introduction, we now transform the optimization problem \eqref{eq:A123} over $\mathcal{F}_2$ in  into another unrestricted, smooth optimization problem over $\R^{d+1}$, so that we may construct functions $h$ in a systematic way with standard numerical optimization methods. For this purpose, we make a couple of helpful observations. First, note that if $h\in\mathcal{F}_2$ then $h\in L^1(\R)$. In fact, by the Cauchy-Schwarz inequality, we have
\begin{align*}
    \int_1^\infty |h(x)|\ \d x &=
    \int_1^\infty |x\, h(x)|\cdot \frac{1}{x}\,  \d x \leq \sqrt{
    	\int_1^\infty x^2|h(x)|^2\,\d x}\cdot \sqrt{ \int_1^\infty \frac{1}{x^2}\,\d x}\,  <\infty.
\end{align*}
Therefore, $\widehat{h}$ is continuous in $\R$, and in particular $\widehat{h}(\pm1/2)=0$. Denoting $I=[-1/2,1/2]$ we have that $\supp \widehat{h}\subset I$. Therefore, by the Stone-Weierstrass theorem we may approximate $\widehat{h}$ uniformly by a polynomial times $\chi_I$.

With the previous observations in mind, we consider functions of the form 
\begin{equation}\label{eq:hPols}
\widehat{h}(x)= \left(
\frac{1}{4}-x^2
\right)g(x)\chi_I(x),
\end{equation}
where 
\[
g(x)=\sum_{i=0}^d a_ix^i \in \R[x]
\]
is a polynomial of degree $d$. Note that the factor $\left(\frac{1}{4}-x^2\right)$ means that $\widehat{h}\left(\pm 1/2\right)=0.$  Denoting $\mathbf{a}=(a_0,a_1,\ldots,a_d)\in\R^{d+1}$, the infimum in \eqref{eq:A123}, restricted to this class, becomes 
\begin{equation}\label{eq:A1d}
    \mathcal{A}_{1,d} :=	\min_{\mathbf{a}\in\R^{d+1}\setminus 0}\dfrac{2 \mathbf{a}\cdot N\mathbf{a}}{\mathbf{a}\cdot D\mathbf{a}},
\end{equation} 
where $N$, $D\in\R^{(d+1)\times(d+1)}$ are defined by 
\begin{equation*}
    N_{ij} = \int_{-\infty}^{\infty}|x|^2f_i(x)\overline{f_j(x)}\,\dx;
    \ \ \ 
   D_{ij}= \int_{-\infty}^{\infty}|x|f_i(x)\overline{f_j(x)}\,\dx;
    \ \ \ 
    f_i(x)=\left[\left(
\frac{1}{4}-y^2
\right)y^i\chi_{I}\right]^\wedge\!\!\!(-x).
\end{equation*}
For all $d\le 20$ and $0\le i\le d$, it is easy to see by direct computation of $f_i$ that $xf_i\in L^2(\R)$, so that $h=\mathbf{a}\cdot (f_0,\ldots,f_d)\in\mathcal{F}_2$ for all $\mathbf{a}\in\R^{d+1}$. The matrices $N$ and $D$ may be computed explicitly for a given $d$, and this is then a smooth optimization problem over $\R^{d+1}$. Solving it numerically for $d=2$, we find
\begin{equation}\label{eq:simpleH}
    \widehat{h_0}(x)=\left(\frac{1}{4}-x^2\right)\left(1-\frac{9}{5}x^2\right) \chi_I(x),
\end{equation}
which yields
\begin{equation}\label{eq:hPold2}
	h_0(x)=\frac{\left(108-25 \pi ^2 x^2\right) \sin (\pi  x)-\pi  x \left(11 \pi ^2 x^2+108\right) \cos (\pi  x)}{40 \pi ^5 x^5}.
\end{equation}
By direct computation in exact rational arithmetic, this gives
\begin{equation*}
	\mathcal{A}_1\le \frac{49484}{38745} = 1.27717...
\end{equation*}
This gives another proof of an upper bound for $\mathcal{A}_1$, which coincides up to four decimal digits with our best bounds. Moreover, using the representation \eqref{9_41pm} we obtain the function in $\mathcal{F}_1$
\begin{equation}\label{eq:f}
	\begin{split}
		f_0(x)= 
		&\frac{P(\pi x)+Q(\pi x)\sin (2 \pi  x)+R(\pi x)\cos(2\pi x)}{738 \pi ^8 x^8}, 
	\end{split}
\end{equation}
where
\begin{eqnarray*}
	P(x) =& \!\!\!\!242x^6+3001x^4+4176,\\ 
	Q(x)  = &\!\!\!\!-242 x^5-576 x^3-11664x,\\
	R(x) = &\!\!\!\!1463 x^4+7488 x^2-5832;
\end{eqnarray*}
see Figure \ref{im:f}.

\begin{figure}[ht] 
\centering
\includegraphics[width=3in]{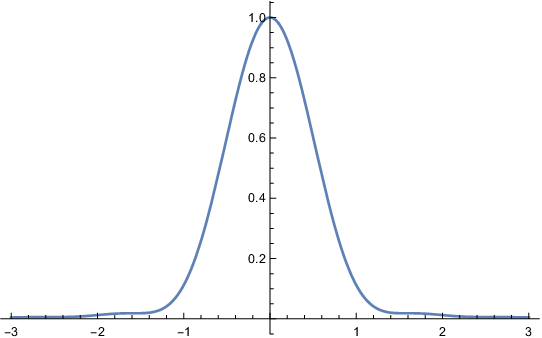} 
\caption{\label{im:f}The function $f_0(x)$ defined in equation \eqref{eq:f}.}
\end{figure}

Additionally, we solve \eqref{eq:A1d} for all $d\le 20$ and observe that, as the degree $d$ increases, the sequence 
$\mathcal{A}_{1,d}$ decreases very slowly, showing only a tiny improvement from $1.27717...$ only in the fifth decimal digit. More precisely, we recover the bound $\mathcal{A}_1\le1.27713505...$ with those much more detailed calculations performed with large degree $d$ of the polynomials $g$. Below, we will show some tables with the results of these computations (see Table \ref{tb:a}), and compare the results with those of our $L^2$-approach from Section \ref{sec:L2}. In Figure \ref{im:hHatPol} and Figure \ref{im:hPol}, we plot the functions $4\widehat h_0$ and $\frac{600}{91} h_0$, respectively, where, since $h_0(0)=\frac{91}{600}$ and $\widehat h_0(0)=\frac{1}{4}$, we renormalized the plots accordingly.

\begin{figure}[ht] 
	\centering
	\begin{minipage}{.5\textwidth}
		\centering
		\includegraphics[width=3in]{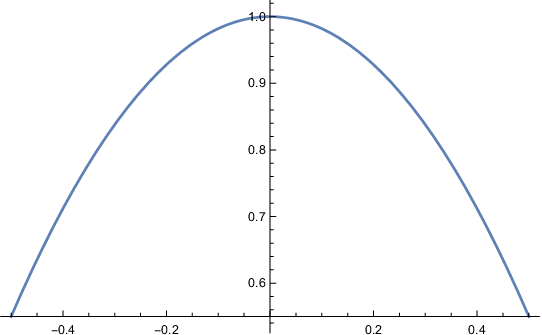} 
		\caption{\label{im:hHatPol}The function  $\widehat{h_0}(x)/\widehat{h_0}(0)$ defined in equation \eqref{eq:simpleH}.}  
	\end{minipage}%
	\begin{minipage}{.5\textwidth}
		\centering
		\includegraphics[width=.9\linewidth]{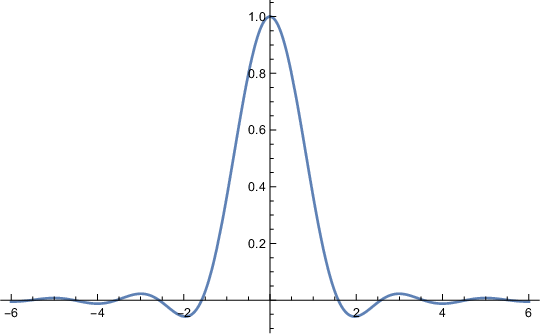} 
		\caption{\label{im:hPol}The function  $h_0(x)/h_0(0)$ defined in equation \eqref{eq:hPold2}.}  
	\end{minipage}
\end{figure}

In Table \ref{tb:a}, we compare the speed of convergence of the two numerical approaches we have presented. The first approach is described above in the present section 
, with functions $h$ defined as in \eqref{eq:hPols}, via polynomials $g$ of some degree $d$. The second approach is the $L^2-$approach described in Section \ref{sec:L2}, with functions $h$ defined by \eqref{eq:hL2}. In both cases, the parameter $d$ is the number of degrees of freedom in the construction of the function $h$. In both cases, the upper bounds for $\mathcal{A}_1$ appear to quickly converge to the first few decimal digits, yet we observe that in the polynomial approach, the bound for $\mathcal{A}_1$ seems to converge much faster to more decimal digits with small values of $d$. Together, all of this gives evidence to the conjecture that the sharp value of $\mathcal{A}_1$, up to its first 9 significant digits, is 
\begin{equation}\label{eq:conj}
 \mathcal{A}_1 =1.277135042...
\end{equation}

Furthermore, the normalized plot of the function $h$ we constructed by using \eqref{eq:hL2} with $d=1000$ is almost indistinguishable from the plot of $h_0$ shown in Figure \ref{im:hPol}. Since the explicit function $h_0$ defined in \eqref{eq:hPold2} already agrees with our conjecture \eqref{eq:conj} to four significant digits, we might expect it to behave close to an extremizer for $\mathcal{A}_1$. Indeed, in Table \ref{tb:zeros}, we compare the first 10 zeros of the functions $h_0$ in \eqref{eq:hPold2} and $h$ in \eqref{eq:hL2} (the latter with $d=1000$). Note that there is a good agreement up to the second decimal digit. We remark that the latter do not change with respect to the values with $d=500$, up to the digits shown, except for a minor change in the last digit of $x_{10}=10.5240...$ (for $d=500$).

\begin{table}[ht]
	\begin{center}
		\begin{tabular}{|c|c||c|c|}
			\hline
			$d$ & $\mathcal{A}_1$ (polynomials) & $d$ &  $\mathcal{A}_1$ ($L^2$)  \\
			\hline \hline
			2 & 1.277171240 &	10 & 1.2771993500 \\ \hline
			4 &  1.277148060 &	50 & 1.2771360175  \\ \hline
			6 & 1.277137688&	100 & 1.2771351946  \\ \hline
			8 & 1.277135865 &	150 & 1.2771350931  \\ \hline
			10 &  1.277135348&	200 & 1.2771350654  \\ \hline
			12 &  1.277135173&	300 & 1.2771350498  \\ \hline
			14 &  1.277135104 &	
    500 & 1.2771350440
   \\ \hline
			16 &1.277135074 &	1000 & 1.2771350424  \\ \hline
			20 & 1.277135052 &	3010 & 1.2771350422  \\ \hline
		\end{tabular}
		\vspace{0.2cm}
		\caption{Comparison of the numerical bounds for $\mathcal{A}_1$ in the polynomial construction of Section \ref{sec:proofNumeric} and in the $L^2-$construction of Section \ref{sec:L2}, as the corresponding parameter $d$ grows.}
		\label{tb:a}
	\end{center}
\end{table}

\begin{table}[ht]
	\begin{center}
		\begin{tabular}{|c||c|c|c|c|c|c|c|c|c|c|}
		\hline
		. &$x_1$ &$x_2$&$x_3$&$x_4$&$x_5$&$x_6$&$x_7$&$x_8$&$x_9$&$x_{10}$  \\ \hline \hline
 Pol & 1.5839 & 2.5715 & 3.5573 & 4.5470 & 5.5395 & 6.5340 & 7.5297 & 8.5264 & 9.5238 & 10.5220\\ \hline 
  $L^2$ & 1.5866 & 2.5648 & 3.5525 & 4.5444 & 5.5387 & 6.5344 & 7.5311 & 8.5284 & 9.5261 & 10.5243\\ \hline 
		\end{tabular}
		\vspace{0.2cm}
		\caption{First positive zeros of the function $h_0$ via polynomials of degree 2 given in \eqref{eq:hPold2}, and via the $L^2-$approach as in \eqref{eq:hL2} with $d=1000.$}
		\label{tb:zeros}
	\end{center}
\end{table}

	\section{Proof of Theorem \ref{2_57pm} } \label{one-deltaproblemwithderivatives}
Let $g\in PW^2$. By Paley-Wiener's theorem, $\widehat{g}$ has compact support in $[-\hh,\hh]$, and using Plancherel's theorem, we obtain
	\begin{align} \label{1_58pm}
		\begin{split} 
		\int_{-\infty}^{\infty}\displaystyle\sum_{n=0}^N\dfrac{a_n}{\pi^{2n}}\big|g^{(n)}(x)\big|^2\,\dx & = \int_{-\frac{1}{2}}^{\frac{1}{2}}\Bigg(\displaystyle\sum_{n=0}^{N}a_n(4t^2)^{n}\Bigg)\big|\widehat{g}(t)\big|^2\,\dt = \int_{-\frac{1}{2}}^{\frac{1}{2}}\mathcal{P}(4t^2)\big|\widehat{g}(t)\big|^2\,\dt.
		\end{split} 
	\end{align}
Since $g\in PW^2$, then
\begin{align}\label{2_58pm}
	g(z)=\int_{-\frac{1}{2}}^{\frac{1}{2}}\widehat{g}(t)\,e^{2\pi izt}\,\dt.
\end{align} 
Then the fact that $g(0)=1$, the positivity of $\mathcal{P}(x)$, the Cauchy-Schwarz inequality and \eqref{1_58pm} yield
\begin{align} \label{2_15pm}
	\begin{split} 
	1=\left|\int_{-\frac{1}{2}}^{\frac{1}{2}}\widehat{g}(t)\,\dt\right|^2 & = \left|\int_{-\frac{1}{2}}^{\frac{1}{2}}\sqrt{\mathcal{P}(4t^2)}\,\widehat{g}(t)\cdot\dfrac{1}{\sqrt{\mathcal{P}(4t^2)}}\,\dt\right|^2  \leq  \left(\int_{-\frac{1}{2}}^{\frac{1}{2}}\mathcal{P}(4t^2)\big|\widehat{g}(t)\big|^2\,\dt\right)\left(\int_{-\frac{1}{2}}^{\frac{1}{2}}\dfrac{1}{\mathcal{P}(4t^2)}\,\dt\right),
\end{split}
\end{align}
which implies \eqref{10_02pm}. Note that equality in \eqref{2_15pm} holds if and only if there is $\lambda\in\C$, such that 
$$
\widehat{g}(t)=\dfrac{\lambda}{\mathcal{P}(4t^2)}
$$
almost everywhere in $[-\hh,\hh]$. Hence, from \eqref{2_58pm} we conclude that
\begin{align*}
	g(z)=\lambda\int_{-\frac{1}{2}}^{\frac{1}{2}}\dfrac{e^{2\pi izt}}{\mathcal{P}(4t^2)}\,\dt = \lambda\int_{0}^{1}\dfrac{\cos{(\pi zt)}}{\mathcal{P}(t^2)}\,\dt.
\end{align*} 
Since $g(0)=1$, then the extremal function is unique and it is is given by \eqref{10_44pm}.

\begin{remark}
Since $\mathcal{P}(x)>0$ for all $x\in [0,1]$, the expression in \eqref{1_58pm} is nonnegative. Thus we obtain a norm in $PW^2$, defined by
$$
\|g\|_{\mathcal{P}}=\left(\int_{-\infty}^{\infty}\displaystyle\sum_{n=0}^N \frac{a_n}{\pi^{2n}} \big| g^{(n)}(x)\big|^2\dx \right)^{\frac{1}{2}}\!\!\!\!,
$$
which can be viewed  as a Sobolev-type norm.
\end{remark}

\section*{Appendix}
Here, we record the following elementary estimates, which are useful in the proof of Theorem \ref{thm:lower_bounds}. Given $0<t<\infty$ and $n\geq 1000$, one has
	\begin{align}
		\label{eq:elem_3}&\sum_{k=n}^\infty\frac{1}{k^2} <  \frac{1}{n^2}+\frac{1}{n}\leq \left(1+\frac{1}{1000}\right)\frac{1}{n} \\
		\label{eq:elem_2} &\sum_{k=1}^n\frac{1}{2k-1}< 1+\log2+\frac{\log n}{2}\leq \frac{3\log n}{4} \\
		\label{eq:elem_1}       &\sum_{k=1}^\infty\frac{\log^2(1+tk)}{k^2} < t^2+\int_0^\infty \frac{\log^2(1+tx)}{x^2}\dx \leq t^2+\frac{\pi^2}{3}t \\                              \label{eq:elem_4} &\sum_{k=n}^\infty\frac{1}{(2k-1)^3} <\frac{1}{(2n-1)^3}+\frac{1}{4(2n-1)^2}\leq \left(\frac{1}{4}+\frac{1}{1999}\right)\frac{1}{(2n-1)^2}	\end{align}

 \section*{Acknowledgements}
 Part of this paper was written while A.C. was a Visiting Researcher in Department of Mathematics at the State University of S\~ao Paulo UNESP. He is grateful for their kind hospitality. 
 We are thankful to Emanuel Carneiro for helpful discussions related to the material of this paper. 


\begin{thebibliography}{9999}
		
\bibitem{Krein} 
N. I. Achieser,
\newblock {Theory of Approximation}, 
\newblock Frederick Ungar Publishing, New York, 1956.
			
		
\bibitem{Boas}
R.~P.~Boas, Jr., 
\newblock Entire Functions, Academic Press, New York, 1954.
			
\bibitem{Brezis}
H. Brezis, 
\newblock Functional Analysis, Sobolev Spaces and Partial Differential Equations, Springer, New York, 2011.
			
			
\bibitem{CMS} E. Carneiro, M. B. Milinovich and K. Soundararajan,
\newblock Fourier optimization and prime gaps,
\newblock Comment. Math. Helv. 94, no. 3 (2019), 533--568.
		
\bibitem{Carneiroteam} E. Carneiro, C. Gonz\'alez-Riquelme, L. Oliveira, A. Olivo, S. Ombrosi, A. P. Ramos and M. Sousa,
\newblock Sharp embeddings between weighted Paley--Wiener spaces,
\newblock preprint at \href{https://arxiv.org/pdf/2304.06442.pdf}{https://arxiv.org/pdf/2304.06442.pdf}


\bibitem{CL2023}
E. Carneiro and F. Littmann, 
\newblock Monotone extremal functions and the weighted Hilbert's inequality, 
\newblock preprint at \href{https://arxiv.org/abs/2302.14658.}{https://arxiv.org/abs/2302.14658.}

		
\bibitem{BCCS}	
A. Chirre, O. F. Brevig, J. Ortega-Cerda and K. Seip, 
\newblock Point evaluation in Paley--Wiener spaces,
\newblock preprint at \href{https://arxiv.org/abs/2210.13922}{https://arxiv.org/abs/2210.13922.}
		
\bibitem{CQH}	
A. Chirre and E. Quesada-Herrera, 
\newblock Fourier optimization and quadratic forms, Q. J. Math. 73 (2022), no. 2, 539--577. 
		
\bibitem{DSB}
R. J. Duffin and A. C. Schaeffer, 
\newblock Some properties of functions of exponential type,
\newblock Bull. Amer. Math. Soc. 44 (1938), no. 4, 236--240.
			
\bibitem{Gor2001}
D. V. Gorbachev, Extremum problem for periodic functions supported in a ball, Mat. Zametki 69 (2001) 346--352.
	
\bibitem{Gor2005}
D. V. Gorbachev, 
\newblock An integral problem of Konyagin and the (C,L)-constants of Nikolaskii, Trudy Inst. Mat. i Mekh. UrO
RAN, 11, (2005), no. 2, 72--91.
	
\bibitem{Gor2018}
D. V. Gorbachev and V. I. Ivanov,
\newblock Tur\'an's and Fej\'er's extremal problems for Jacobi transform, Anal. Math. 44 (2018), 419--432.
	
\bibitem{Gor2019} 
D. V. Gorbachev and V. I. Ivanov, 
\newblock Tur\'{a}n, Fej\'er and Bohman extremal problems for the multivariate Fourier transform in terms of the eigenfunctions of a Sturm--Liouville problem, Sb. Math., 210:6 (2019), 809--835.
	

\bibitem{GorM} 
D. V. Gorbachev and I.A. Martyanov, 
\newblock On interrelation of Nikolskii constants for trigonometric polynomials and entire functions of exponential type, Chebyshevskii Sb., 19:2 (2018), 80--89.

		
\bibitem{HL}
G.~H.~Hardy and J.~E.~Littlewood, 
\newblock Some integral inequalities connected with the calculus of variations, The Quarterly Journal of Mathematics, 3 (1932), 241--252.

\bibitem{HV} 
J.~Holt and J.~D.~Vaaler,
\newblock The Beurling-Selberg extremal functions for a ball in the Euclidean space,
\newblock Duke Math. Journal 83 (1996), 203--247.

\bibitem{HB} 
L. H\"ormander and B. Bernhardsson,
\newblock An extension of Bohr’s inequality
\newblock Boundary value problems for partial differential equations and applications, RMA Res. Notes Appl. Math., 29 (1993), 179--194.



\bibitem{arb}
F. ~Johansson, 
\newblock Arb: efficient arbitrary-precision midpoint-radius interval arithmetic, IEEE Transactions on Computers, 66 (2017), no. 8, 1281--1292.

		
\bibitem{Korevaar}
J. Korevaar, 
\newblock An inequality for entire functions of exponential type,
\newblock Nieuw Arch. Wiskunde (2) 23 (1949), 55--62.
		
		
\bibitem{Sz2015} 
S. Krenedits and S. Gy. R\'ev\'esz, 
\newblock Carath\'eodory Fej\'er type extremal problems on locally compact Abelian groups, J. Approx. Theory 194 (2015), 108--131.
		
\bibitem{Landau} 
E. Landau, 
\newblock Ungleichungen f\"ur zweimal differenzierbare Funktionen, 
\newblock Proc. London Math. Soc. 13 (1913), 43--49.


\bibitem{Niko}
S. M. Nikolskii, 
\newblock Approximation of functions
of several variables and imbedding theorems, 
\newblock Berlin; Heidelberg; New
York: Springer, 1975.
		
  
\bibitem{Sz2011}
S. Gy. R\'ev\'esz, 
\newblock Tur\'an extremal problems on locally compact abelian groups, Anal. Math. 37 (2011), 15--50.
		
\bibitem{StW} 
E. Stein and G. Weiss, 
\newblock Introduction to Fourier analysis on Euclidean spaces, Princeton University Press, Princeton, 1971. 
		
\bibitem{G} 
G. Szeg\"o,
\newblock Orthogonal polynomials, 
\newblock Fourth edition. American Mathematical Society, Colloquium Publications, Vol. XXIII. American Mathematical Society, Providence, R.I., 1975.
	

\end{thebibliography}
\end{document}